\documentclass[12pt,reqno]{amsart}
\usepackage{amssymb}
\usepackage[latin1]{inputenc}
\usepackage{dsfont}
\usepackage{hyperref}
\usepackage{fullpage}

\newtheorem{theorem}{Theorem}[section]

\newtheorem{proposition}[theorem]{Proposition}
\newtheorem{corollary}[theorem]{Corollary}
\theoremstyle{definition}
\newtheorem{definition}[theorem]{Definition}

\theoremstyle{remark}
\newtheorem{remark}[theorem]{Remark}

\numberwithin{equation}{section}

\begin{document}

\title[On some properties of a universal sigma-finite measure associated with 
a remarkable class of submartingales]
{On some properties of a universal sigma-finite measure associated with 
a remarkable class of submartingales}
\author[J. Najnudel]{Joseph Najnudel}
\address{Institut f\"ur Mathematik, Universit\"at Z\"urich, Winterthurerstrasse 190,
8057-Z\"urich, Switzerland}
\email{\href{mailto:joseph.najnudel@math.uzh.ch}{joseph.najnudel@math.uzh.ch}}

\author[A. Nikeghbali]{Ashkan Nikeghbali}
\email{\href{mailto:ashkan.nikeghbali@math.uzh.ch}{ashkan.nikeghbali@math.uzh.ch}}

\date{\today}

\begin{abstract}
In a previous work, we associated with any submartingale $X$ of class $(\Sigma)$, defined on a filtered probability
space $(\Omega, \mathcal{F}, \mathbb{P}, (\mathcal{F}_t)_{t \geq 0})$ satisfying some technical conditions, 
a $\sigma$-finite measure $\mathcal{Q}$ on $(\Omega, \mathcal{F})$, such that for all 
$t \geq 0$, and for all events $\Lambda_t
 \in \mathcal{F}_t$: $$ \mathcal{Q} [\Lambda_t, g\leq t] = \mathbb{E}_{\mathbb{P}} [\mathds{1}_{\Lambda_t} X_t],$$
where $g$ is the last hitting time of zero by the process $X$. The measure $\mathcal{Q}$,
which was previously studied in particular cases related with Brownian penalisations and problems
 in mathematical finance, enjoys some remarkable properties which are detailed in this paper. 
Most of these properties are related to a certain class of nonnegative martingales, defined 
as the local densities (with respect to $\mathbb{P}$) of the finite measures which are absolutely continuous
with respect to $\mathcal{Q}$. In particular, we obtain a decomposition of any nonnegative supermartingale 
into three parts, one of them being a martingale in the class described above. If the initial supermartingale
is a martingale, this decomposition
corresponds to the 
decomposition of finite measures on $(\Omega, \mathcal{F})$ as sums of three measures, such that 
the first one
is absolutely continuous with respect to $\mathbb{P}$, the second one is absolutely continuous with respect to 
$\mathcal{Q}$ and the third one is singular with respect to $\mathbb{P}$ and $\mathcal{Q}$. From the properties of the measure $\mathcal{Q}$, we also deduce a universal class of penalisation results of the probability measure $\mathbb{P}$ with a large class of functionals: the measure $\mathcal{Q}$ appears to be the unifying object in these problems. 
\end{abstract}

\maketitle
\section*{Notation}
In this paper, $(\Omega,\mathcal{F},(\mathcal{F}_t)_{t \geq 0},\mathbb{P})$ will denote a filtered probability space. $\mathcal{C}(\mathbb{R}_+,\mathbb{R})$ is the space of continuous functions from $\mathbb{R}_+$ to $\mathbb{R}$. $\mathcal{D}(\mathbb{R}_+,\mathbb{R})$ is the space of c\`adl\`ag functions from $\mathbb{R}_+$ to $\mathbb{R}$. If  $Y$ is a random variable, we denote indifferently by $\mathbb{P}[Y]$ or by $\mathbb{E}_{\mathbb P}[Y]$
the expectation of $X$ with respect to $\mathbb{P}$. If $(A_t)_{t\geq0}$ is an increasing process, as usual, the increasing limit of $A_t$, when $t\to\infty$, is denoted $A_{\infty}$.

\section{Introduction}
In a paper by Madan, Roynette and Yor \cite{MRY}, and a set of lectures by Yor \cite{BeYor}, the
authors prove that if $(M_t)_{t \geq 0}$ is a continuous nonnegative local martingale  defined
 on a filtered probability
 space $(\Omega,\mathcal{F},(\mathcal{F}_t)_{t \geq 0},\mathbb{P})$ satisfying the usual assumptions, and 
such that  $\lim_{t\to\infty}M_t=0$, then for any $K \geq 0$:
\begin{equation}\label{BS}
K \, \mathbb{P} \left[ F_t \, \mathds{1}_{g_K \leq t} \right] = \mathbb{P} \left[F_t (K-M_t)_+ \right],	
	\end{equation} where $g_K=\sup\{t \geq 0: M_t=K\}$. The formula (\ref{BS}), which represents
 the price of a European put option in terms of the probability distribution of some last passage time 
gives, in a particular case, a positive answer to the following problem, also stated in \cite{BeYor}
and \cite{MRY}: for which submartingales $X$ can we find
 a $\sigma$-finite measure $\mathcal{Q}$ and the end of an optional set $g$ such that
\begin{equation}\label{masterequation2}
\mathcal{Q} \left[ F_t \, \mathds{1}_{g \leq t} \right] = \mathbb{P} \left[F_t X_t \right]?
\end{equation}
This problem was previously encountered in the literature in different situations.  
In \cite{AY1}, Az\'ema and Yor prove that for any continuous and uniformly integrable 
 martingale $M$, (\ref{masterequation2}) holds for $X_t=|M_t|$, $\mathcal{Q}=|M_\infty|.\mathbb{P}$ and
 $g=\sup\{t \geq 0: M_t=0\}$, or equivalently $$|M_t|=\mathbb{P}[|M_\infty|\mathds{1}_{g\leq t}|
\mathcal{F}_t].$$ Here again the measure $\mathcal{Q}$ is finite. A particular case
 where the measure $\mathcal{Q}$ is not finite was obtained by  Najnudel, Roynette and Yor
in their study of Brownian penalisations (see \cite{NRY}). For example, they prove the existence of
 the measure $\mathcal{Q}$ when $X_t=|W_t|$ is the absolute value of the standard
 Brownian Motion. In this case, the measure $\mathcal{Q}$ is not finite but only $\sigma$-finite
 and is singular with respect to the Wiener measure: it satisfies $\mathcal{Q}(g=\infty)=0$, where
 $g=\sup\{t \geq 0:\;W_t=0\}$. Now, the existence of $\mathcal{Q}$ in all the examples cited above
is a consequence of a general result proved by the authors of the present paper in \cite{NN1}.
 The relevant class of submartingales is called $(\Sigma)$, it was first introduced by Yor in \cite{Y}
and some of its main properties were further studied in \cite{N}. Let us recall its definition.
\begin{definition}[\cite{N,Y}]
Let $(\Omega,\mathcal{F},(\mathcal{F}_t)_{t \geq 0},\mathbb{P})$ be a filtered probability space. A 
nonnegative submartingale (resp. local submartingale) $(X_t)_{t \geq 0}$ is of class $(\Sigma)$, iff it can 
be decomposed as
$X_t = N_t + A_t$ where $(N_t)_{t \geq 0}$ and $(A_t)_{t \geq 0}$ are $(\mathcal{F}_t)_{t \geq 0}$-adapted 
processes satisfying the following assumptions:
\begin{itemize}
\item $(N_t)_{t \geq 0}$ is a c\`adl\`ag martingale (resp. local martingale);
\item $(A_t)_{t \geq 0}$ is a continuous increasing process, with $A_0 = 0$;
\item The measure $(dA_t)$ is carried by the set $\{t \geq 0, X_t = 0 \}$.
\end{itemize}
\end{definition}
\noindent
 One notes that a  process of class $(\Sigma)$ is "almost" a martingale: outside the zeros of $X$, the 
process $A$ does not increase. In fact many processes one often encounters fall into this
 class, e.g. $X_t=|M_t|$ where $(M_t)_{t \geq 0}$ is a continuous local martingale,  $X_t=(M_t-K)_+$ where 
$(M_t)_{t \geq 0}$ is a c\`adl\`ag local martingale with only positive 
jumps and $K\in\mathbb{R}$ is a constant, $X_t=S_t-M_t$ where $(M_t)_{t \geq 0}$ is a local martingale
 with only negative jumps and $S_t=\sup_{u\leq t} M_u$. Other remarkable families of examples consist of a large class of recurrent diffusions on natural scale (such as some powers of Bessel processes of dimension $\delta\in(0,2)$, see \cite{NN1}) or of a function of a symmetric L\'evy process; in these cases, $A_t$  is the local time of the diffusion process or of the L\'evy process.

Note that in the case where $A_{\infty} = \infty$, $\mathbb{P}$-almost surely (this condition holds
if  $(X_t)_{t \geq 0}$ is a reflected Brownian motion), and
$(\Omega,\mathcal{F},(\mathcal{F}_t)_{t \geq 0},\mathbb{P})$ 
satisfies the usual conditions, the measure $\mathcal{Q}$ cannot exist: otherwise, we would have for all 
$t \geq 0$,
$$ \mathbb{P} [X_t] = \mathbb{P} [X_t  \mathds{1}_{g > t}] = \mathcal{Q} [g \leq t, g > t] = 0,$$
since the event $\{g > t\}$ is $\mathbb{P}$-almost sure, and then in $\mathcal{F}_t$. Hence, 
$X$ would be indistinguishable from zero, which contradicts the fact that $A_{\infty} = \infty$. 
This issue explains why usual conditions are not assumed in the sequel of this paper. On the other hand, 
we also encounter some problems if we do not complete the probability spaces: for example, 
if $\Omega = \mathcal{C}(\mathbb{R}_+, \mathbb{R})$, $\mathcal{F}_t$ is the $\sigma$-algebra
generated by the canonical process $X$ up to time $t$, and $\mathbb{P}$ is Wiener measure,
 then there does not exist a c\`adl\`ag  and $(\mathcal{F}_t)_{t \geq 0}$-adapted version of the local time
which is well-defined everywhere (and not only $\mathbb{P}$-almost surely). In order to avoid also 
this technical problem, we assume that the filtration satisfies some particular conditions, intermediate between 
the right-continuity and the usual conditions. These assumptions, called "natural conditions",
 were first introduced by Bichteler
in \cite{B}, and then rediscovered  in \cite{NN2} (there they are also called N-usual conditions) where it is proved that most of the properties which generally hold
under the usual conditions remain valid under the natural conditions (for example, existence of c\`adl\`ag versions
of martingales, the Doob-Meyer decomposition, the d\'ebut theorem, etc.). Let us recall here the definition.
\begin{definition} \label{natural}
A filtered probability space $(\Omega,\mathcal{F}, (\mathcal{F}_t)_{t \geq 0}, \mathbb{P})$,
  satisfies the natural conditions iff the two following assumptions hold:
\begin{itemize}
\item The filtration $(\mathcal{F}_t)_{t \geq 0}$ is right-continuous;
\item For all $t \geq 0$, and for every $\mathbb{P}$-negligible set $A \in \mathcal{F}_t$, all
the subsets of $A$ are contained in $\mathcal{F}_0$.
\end{itemize}
\end{definition}
\noindent
This definition is slightly different from the definitions given in \cite{B} and \cite{NN2} but one can 
easily check that it is equivalent. 
The natural enlargement of a filtered probability space can be defined by using the following proposition:
\begin{proposition}[\cite{NN2}]
Let $(\Omega, \mathcal{F}, (\mathcal{F}_t)_{t \geq 0}, \mathbb{P})$ be a filtered probability space.
There exists a unique filtered probability space $(\Omega, \widetilde{\mathcal{F}}, 
 (\widetilde{\mathcal{F}}_t)_{t \geq 0}, \widetilde{\mathbb{P}})$ (with the same set $\Omega$), such that:
\begin{itemize}
\item For all $t \geq 0$, $\widetilde{\mathcal{F}}_t$ contains $\mathcal{F}_t$, $\widetilde{\mathcal{F}}$ 
contains $\mathcal{F}$ and $\widetilde{\mathbb{P}}$ is an extension of $\mathbb{P}$;
\item The space $(\Omega, \widetilde{\mathcal{F}}, 
 (\widetilde{\mathcal{F}}_t)_{t \geq 0}, \widetilde{\mathbb{P}})$ satisfies the natural conditions;
\item For any filtered probability space $(\Omega, \mathcal{F}', (\mathcal{F}'_t)_{t \geq 0}, \mathbb{P}')$
satisfying the two items above, $\mathcal{F}'_t$ contains $\widetilde{\mathcal{F}}_t$ for all $t \geq 0$,
$\mathcal{F}'$ contains $\widetilde{\mathcal{F}}$ and $\mathbb{P}'$ is an extension of 
$\widetilde{\mathbb{P}}$.  
\end{itemize}
\noindent
The space $(\Omega, \widetilde{\mathcal{F}}, 
 (\widetilde{\mathcal{F}}_t)_{t \geq 0}, \widetilde{\mathbb{P}})$ is called the natural
 enlargement of $(\Omega, \mathcal{F},  (\mathcal{F}_t)_{t \geq 0},  \mathbb{P})$.
  \end{proposition}
\noindent
Intuitively, the natural enlargement of a filtered probability space is its smallest extension which 
satisfies the natural conditions. 
We also introduce a class of filtered measurable spaces $(\Omega, \mathcal{F},
  (\mathcal{F}_t)_{t \geq 0})$ such that any compatible family $(\mathbb{Q}_t)_{t \geq 0}$ 
of probability measures, $\mathbb{Q}_t$ defined on $\mathcal{F}_t$, can be extented to a probability 
measure $\mathbb{Q}$ defined on $\mathcal{F}$. 
\begin{definition}
 \label{P}
Let $(\Omega, \mathcal{F}, (\mathcal{F}_t)_{t \geq 0})$ be a filtered measurable space, such that
$\mathcal{F}$ is the $\sigma$-algebra generated by $\mathcal{F}_t$, $t \geq 0$: 
$\mathcal{F}=\bigvee_{t\geq0}\mathcal{F}_t$. We say that the property\footnote{(P) stands for Parthasarathy since such conditions where introduced by him in \cite{Parth}.} (P)
 holds if and only if $(\mathcal{F}_t)_{t \geq 0}$ enjoys the following properties: 
\begin{itemize}
\item for all $t \geq 0$, $\mathcal{F}_t$ is generated by a countable number of sets;
\item for all $t \geq 0$, there exists a Polish space $\Omega_t$, and a surjective map 
 $\pi_t$ from $\Omega$ to $\Omega_t$, such that $\mathcal{F}_t$ is the $\sigma$-algebra of the inverse
 images, by $\pi_t$, of Borel sets in $\Omega_t$, and such that for all $B \in \mathcal{F}_t$, 
 $\omega \in \Omega$, $\pi_t (\omega) \in \pi_t(B)$ implies $\omega \in B$;
\item if $(\omega_n)_{n \geq 0}$ is a sequence of elements of $\Omega$, such that for all $N \geq 0$,
$$\bigcap_{n = 0}^{N} A_n (\omega_n) \neq \emptyset,$$
where $A_n (\omega_n)$ is the intersection of the sets in $\mathcal{F}_n$ containing $\omega_n$, 
then:
$$\bigcap_{n = 0}^{\infty} A_n (\omega_n) \neq \emptyset.$$
\end{itemize}
\end{definition}
\noindent
A fundamental example of a filtered measurable space $(\Omega, \mathcal{F}, (\mathcal{F}_t)_{t \geq 0})$
satisfying the property (P) can be constructed as follows: we take 
 	 $\Omega$ to be equal to  $\mathcal{C}(\mathbb{R}_+,\mathbb{R}^d)$, the space of continuous functions from 
	$\mathbb{R}_+$ to $\mathbb{R}^d$, or $\mathcal{D}(\mathbb{R}_+,\mathbb{R}^d)$, the space of c\`adl\`ag 
functions from $\mathbb{R}_+$ 
	to $\mathbb{R}^d$ (for some $d \geq 1$), and for $t \geq 0$, we define
 $(\mathcal{F}_{t})_{t \geq 0}$ as the natural filtration of the canonical
 process, and we set $$\mathcal{F} := \bigvee_{t\geq0}\mathcal{F}_t.$$
The combination of the property (P) and the natural conditions gives the following definition:
\begin{definition} \label{NP}
Let $(\Omega, \mathcal{F},(\mathcal{F}_t)_{t \geq 0}, \mathbb{P})$ be a filtered probability space. 
We say that it satisfies the property (NP) if and only if it is the natural enlargement of a 
filtered probability space $(\Omega, \mathcal{F}^0,(\mathcal{F}^0_t)_{t \geq 0}, \mathbb{P}^0)$
such that the filtered measurable space $(\Omega, \mathcal{F}^0,(\mathcal{F}^0_t)_{t \geq 0})$ enjoys 
property (P). 
\end{definition}
\noindent
In \cite{NN2} the following result on extension of probability measures is proved (in a slightly 
more general form): 
\begin{proposition} \label{extensionaugmentation}
Let $(\Omega, \mathcal{F}, (\mathcal{F}_t)_{t \geq 0}, \mathbb{P})$ be a filtered probability space, 
satisfying property (NP). 
 Then, the $\sigma$-algebra $\mathcal{F}$ is the $\sigma$-algebra generated by $(\mathcal{F}_t)_{t \geq 0}$,
and for all coherent families of probability measures
$(\mathbb{Q}_t)_{t \geq 0}$ such that $\mathbb{Q}_t$ is defined 
on $\mathcal{F}_t$,
and is absolutely continuous with respect to the restriction
of $\mathbb{P}$ to $\mathcal{F}_t$, there exists a unique probability measure
$\mathbb{Q}$
on $\mathcal{F}$ which coincides with $\mathbb{Q}_t$ on $\mathcal{F}_t$ 
for all $t \geq 0$. 
\end{proposition}
\noindent
By using all the results and definitions above, one can state rigorously the main result
 of \cite{NN1} in its most general form:
\begin{theorem} \label{all}
Let $(X_t)_{t \geq 0}$ be a submartingale of the class $(\Sigma)$ (in particular $X_t$ is integrable 
for all $t \geq 0$), defined on a filtered probability space 
$(\Omega, \mathcal{F}, \mathbb{P}, (\mathcal{F}_t)_{t \geq 0})$ which satisfies the property (NP).
In particular, $(\mathcal{F}_t)_{t \geq 0}$
 satisfies the natural conditions and $\mathcal{F}$ is 
the $\sigma$-algebra generated by $\mathcal{F}_t$, $t \geq 0$. 
Then, there exists a unique $\sigma$-finite measure $\mathcal{Q}$, defined on
 $(\Omega, \mathcal{F}, \mathbb{P})$,
such that for $g:= \sup\{t \geq 0, X_t = 0 \}$:
\begin{itemize}
\item $\mathcal{Q} [g = \infty] = 0$;
\item For all $t \geq 0$, and for all $\mathcal{F}_t$-measurable, bounded random variables $F_t$,
\begin{equation}\label{nikomuk}
\mathcal{Q} \left[ F_t \, \mathds{1}_{g \leq t} \right] = \mathbb{P} \left[F_t X_t \right].
\end{equation}
\end{itemize}
\noindent
\end{theorem} 
\noindent As already mentioned, Theorem \ref{all} has already been obtained in some special cases (but 
not under the most rigorous formulation with the correct assumption on the underlying filtered
 probability space) such as the case where $X_t$ is the absolute value of the canonical 
process on the Wiener space or when $X_t=|Y_t|^{\alpha-1}$ where $Y$ is a symmetric
 stable L\'evy process of index $\alpha\in(1,2)$, although in this latter case the
 property (\ref{nikomuk}) was not noticed (\cite{YYY}). In fact, almost all our results will
 apply to a large class of symmetric L\'evy processes including the symmetric stable L\'evy processes
 of index $\alpha\in(1,2)$. We shall now detail a little more this last example since it provides
 natural examples of processes with jumps,
 living on the Skorokhod space. Let us define, one the space 
$\mathcal{D}(\mathbb{R}_+,\mathbb{R})$ 
of c\`adl\`ag functions from $\mathbb{R}_+$ to $\mathbb{R}$, 
 $(\mathcal{F}_{t})_{t \geq 0}$ as the natural filtration of the canonical
 process $(Y_t)_{t \geq 0}$, and let us set $$\mathcal{F} := \bigvee_{t\geq0}\mathcal{F}_t.$$
 We consider on
 $\mathcal{D}(\mathbb{R}_+,\mathbb{R})$ the probability $\mathbb{P}$ under which $(Y_t)_{t \geq 0}$ is
 a symmetric L\'evy process
 with exponent $\Psi$: $$\mathbb{P}[\exp(i \xi Y_t)]=\exp(-t\Psi(\xi)).$$ 
Moreover, we assume that $0$ is regular for itself and that $(Y_t)_{t \geq 0}$ is recurrent, or equivalently
 (see Bertoin \cite{Bertoin}):
 $$\int_{-\infty}^{\infty}\dfrac{d \xi}{1+\Psi(\xi)}<\infty \text{ and }
 \int_{0}\dfrac{d \xi}{\Psi(\xi)}=\infty.$$
 In the case where $(Y_t)_{t \geq 0}$ is a symmetric $\alpha$-stable L\'evy process of index
 $\alpha\in(1,2)$, $\Psi(\xi)=|\xi|^\alpha$ and the above conditions on $\Psi$ are
 satisfied (see \cite{Bertoin}). Salminen and Yor \cite{SalYor} have 
proved that if for some $x \in \mathbb{R}$:
 $$v(x)=\dfrac{1}{\pi}\int_{0}^\infty \dfrac{1-\cos(\xi x)}{\Psi(\xi)} d \xi,$$
 then $$v(Y_t-x)=v(x)+N_t^x+L_t^x,$$where $N_t^x$ is a martingale and
 where $(L_t^x)_{t \geq 0}$ is the local time at level $x$ of the L\'evy process $(Y_t)_{t \geq 0}$. Since
  $(L_t^x)_{t \geq 0}$ is continuous, increasing, adapted
 and only increases when $Y_t=x$ (see Bertoin \cite{Bertoin} chapter V), the process $(v(Y_t-x))_{t \geq 0}$ is
 of class $(\Sigma)$, moreover, $(Y_t)_{t \geq 0}$ is recurrent and $0$ is regular
 for itself, which implies that $\lim_{t\to\infty} L_t^x=\infty$, $\mathbb{P}$-almost surely.  Hence
 Theorem \ref{all} applies, and for any $x\in\mathbb{R}$, there exists a $\sigma$-finite measure 
$\mathcal{Q}_x$, singular to $\mathbb{P}$ and such that all the properties 
of Theorem \ref{all} are satisfied with $X_t=v(Y_t-x)$ and $g\equiv g_x=\sup\{t: Y_t=x\}$. In the special
 case of symmetric $\alpha$-stable L\'evy processes of index $\alpha\in(1,2)$, $v(x)=c(\alpha) |x|^{\alpha-1}$
 for some explicit constant $c(\alpha)$ (see \cite{SalYor}). In the sequel, all
 our results which do not require any assumptions on the sign of the jumps
 will apply to this family of examples as well.

Let us now shortly recall the general construction of $\mathcal{Q}$ given in \cite{NN1}.
For a Borel, integrable, strictly positive and bounded function $f$ from $\mathbb{R}$ to $\mathbb{R}$, 
one  defines the function $G$ as
$$ G(x) = \int_{x}^{\infty} f(y) \, dy,$$
and then one proves 
that the process \begin{equation}
 \left( M^f_t := G(A_t) - \mathbb{P} [G(A_{\infty})| \mathcal{F}_t] +
f(A_t) X_t \right)_{t \geq 0}, \label{MTF2}
\end{equation}
\noindent
is a martingale with respect to $\mathbb{P}$ and the filtration $(\mathcal{F}_t)_{t \geq 0}$.
Since $(\Omega, \mathcal{F}, \mathbb{P}, (\mathcal{F}_t)_{t \geq 0})$ satisfies the natural conditions
and since $G(A_t) \geq G(A_{\infty})$, one can suppose that this martingale is nonnegative and c\`adl\`ag, by 
choosing carefully the 
version of $\mathbb{P} [G(A_{\infty})|\mathcal{F}_t]$. In this case, since 
$(\Omega, \mathcal{F}, \mathbb{P}, (\mathcal{F}_t)_{t \geq 0})$ satisfies the property (NP), there 
exists a unique finite measure $\mathcal{M}^f$ such that for all $t \geq 0$, and for all bounded,
 $\mathcal{F}_t$-measurable functionals $\Gamma_t$: 
$$\mathcal{M}^f [\Gamma_t] = \mathbb{P} [\Gamma_t M_t^f].$$
Now, since $f$ is strictly positive, one can define a $\sigma$-finite measure $\mathcal{Q}^f$ by:
$$\mathcal{Q}^f := \frac{1}{f(A_{\infty})} \, . \mathcal{M}^f.$$
It is proved in \cite{NN1} that if the function $G/f$ is unformly bounded (this condition 
is, for example, satisfied for $f (x) = e^{-x}$), then $\mathcal{Q}^f$ satisfies the conditions
defining $\mathcal{Q}$ in Theorem \ref{all}, which implies the existence part of this result. 
The uniqueness part is proved just after in a very easy way: one remarkable consequence 
of the uniqueness is the fact that $\mathcal{Q}^f$ does not depend on the choice of $f$.

One of the remarkable features of the measure $\mathcal{Q}$, in the special case of the Wiener space
 and when $(X_t)_{t \geq 0}$ is the absolute value of the Wiener process, is that it allows a unified
 view of some penalisation problems related with Wiener measure. More precisely,  Roynette, Vallois and Yor (\cite{RVY}) consider $\mathbb{W}$, the Wiener measure 
on the space  $\mathcal{C}(\mathbb{R}_+, \mathbb{R})$
endowed with its canonical filtration 
$(\mathcal{F}_s)_{s \geq 0}$ (not completed), and then they  define the $\sigma$-algebra $\mathcal{F}$
by $$\mathcal{F}:= 
\underset{s \geq 0}{\bigvee} \mathcal{F}_s.$$ 
They consider
 $(\Gamma_t)_{t \geq 0}$,  a family of 
nonnegative random variables on the same space, such that 
$$0 < \mathbb{W} [\Gamma_t]  < \infty,$$
and for $t \geq 0$, they define the probability measure $$\mathbb{Q}_t := 
\frac{\Gamma_t}{\mathbb{W}[\Gamma_t]} \, . \mathbb{W}.$$
 Then they are able to prove that for many examples of families of functionals $(\Gamma_t)_{t \geq 0}$, there exists a probability 
measure $\mathbb{Q}_{\infty}$ which can be considered as the weak limit of $(\mathbb{Q}_t)_{t \geq 0}$ 
when $t$ goes to infinity, in the following sense: for all $s \geq 0$ and for all events 
$\Lambda_s \in \mathcal{F}_s$,
$$\mathbb{Q}_t [\Lambda_s] \underset{t \rightarrow \infty}{\longrightarrow} \mathbb{Q}_{\infty} [\Lambda_s].$$ Finding the measure $\mathbb{Q}_{\infty}$ amounts
 to solving the penalisation problem associated with the functional $(\Gamma_t)_{t\geq0}$. The
 functional $\Gamma_t$ is typically some function of the local time or the running supremum of the
 Wiener process, or some Feynman-Kac functional of the Wiener process. In the
 monograph \cite{NRY}, Najnudel, Roynette and Yor have proved that the measure $\mathcal{Q}$ associated
 with the absolute value of the Wiener process allows a unified approach to many
 of the examples dealt with separately  in the literature: under some technical
 conditions on the functionals $(\Gamma_t)_{t \geq 0}$, they show that the measure $\mathbb{Q}_{\infty}$ is absolutely continuous with respect to $\mathcal{Q}$ with an explicit density. In this paper, we shall completely solve the penalisation problem (under the assumptions of Theorem \ref{all}) with functionals of the form $\Gamma_t=F_t X_t$, where  $F_t$ is some functional satisfying some not so restrictive condition. In particular we need no assumption on the continuity of the paths of $(X_t)$, nor any Markov or scaling properties.

More precisely, throughout this paper, we establish some of the fundamental properties of the
 measure $\mathcal{Q}$ (which also prepare the ground for a forthcoming work on penalisation
 of diffusion paths).   A remarkable class of martingales defined 
as local densities (with respect to $\mathbb{P}$) of finite measures, absolutely continuous with 
respect to $\mathcal{Q}$,
is involved in a crucial way. The precise definition of these martingales is given in
 Section \ref{sec::2}, 
and they are explicitly computed in some particular cases. In Section \ref{sec::3}, we 
study their behaviour when $t$ goes to infinity, in the most interesting  case where 
$A_{\infty} = \infty$, $\mathbb{P}$-almost surely, and we deduce some information about the 
behaviour of $(X_t)_{t \geq 0}$ under the measure $\mathcal{Q}$. We shall then naturally deduce
 the announced universal penalisation results from our study of the asymptotic behaviour
 of these martingales  and  of $(X_t)_{t \geq 0}$  under $\mathcal{Q}$. In Section \ref{sec::4}, 
 we give a new decomposition of any nonnegative supermartingale 
into the sum of three nonnegative terms, such that the first one is a uniformly integrable
martingale, and the second is a martingale in the class described above. If the initial supermartingale
is a martingale, this decomposition can be interpreted as the decomposition of a
finite measure on $(\Omega, \mathcal{F})$ as the sum of a three measures, one of them being
absolutely continuous with respect to $\mathbb{P}$, the second one being absolutely continuous with 
respect to $\mathcal{Q}$, and the last one being singular with respect to $\mathbb{P}$ and $\mathcal{Q}$. On our
 way we shall also establish the following remarkable fact: if $X$ is of class $(\Sigma)$, with only
 positive jumps, and if $A_\infty=\infty$, then for any $a\in\mathbb{R}$, $(X_t-a)_+$ is of
 class $(\Sigma)$ and the measure $\mathcal{Q}_a$ associated with it is the same as the measure
 $\mathcal{Q}$ associated with $X$. This invariance property was observed in \cite{NRY} in the
 special case where $X$ is the absolute value of the Wiener process.

\section{A remarkable class of martingales related to the measure $\mathcal{Q}$} \label{sec::2}
Let us first remark that since $\mathcal{Q} [g= \infty] = 0$,
 one has $$\mathcal{Q} [A_{\infty}= \infty] = 0,$$ i.e. 
$A_{\infty}$ is finite $\mathcal{Q}$-almost everywhere. Let us state a useful result which was proved in \cite{NN1}:
\begin{proposition} \label{f}
Let $f$ be an integrable function from $\mathbb{R}_+$ to $\mathbb{R}_+$.
Then under the assumptions of Theorem \ref{all}, the measure
$$\mathcal{M}^f := f(A_{\infty}). \mathcal{Q}$$
 is the unique positive, finite measure such that for all $t \geq 0$, and 
for all bounded, $\mathcal{F}_t$-measurable functionals $\Gamma_t$:
\begin{equation}
\mathcal{M}^f[\Gamma_t] = \mathbb{P} [\Gamma_t M^f_t], \label{hoi}
\end{equation}
where the process $(M_t^f)_{t \geq 0}$ is given by:
$$M_t^f := G(A_t) - \mathbb{P} [G(A_{\infty})| \mathcal{F}_t] + f(A_t)X_t$$
for $$G(x) := \int_x^{\infty} f(y) dy.$$
In particular, $(M_t^f)_{t \geq 0}$ is a martingale, c\`adl\`ag if one chooses a suitable
version of the conditional expectation of $G(A_{\infty})$ given $\mathcal{F}_t$. Moreover, 
$(M_t^f)_{t \geq 0}$ is uniquely determined by $f$ in the following sense: two c\`adl\`ag
martingales satisfying (\ref{hoi}) are necessarily indistinguishable. 
\end{proposition} 
 \noindent
 Proposition \ref{f} gives a relation between a finite measure which is absolutely continuous with 
respect to $\mathcal{Q}$ ($\mathcal{M}^f$), and a c\`adl\`ag martingale $(M_t^f)_{t \geq 0}$. 
This relation can be generalized as follows:
\begin{proposition} \label{MtF}
We suppose that the assumptions of Theorem \ref{all} hold and we take the same notation. 
Let $F$ be a $\mathcal{Q}$-integrable, nonnegative functional defned on $(\Omega, \mathcal{F})$. Then, there 
exists
 a c\`adl\`ag $\mathbb{P}$-martingale $(M_t(F))_{t \geq 0}$ such
 that the measure $\mathcal{M}^F := F. \mathcal{Q}$ is the unique finite measure
 satisfying, for all $t \geq 0$, and 
for all bounded, $\mathcal{F}_t$-measurable functionals $\Gamma_t$:
$$\mathcal{M}^F [\Gamma_t] = \mathbb{P} [\Gamma_t M_t(F)].$$
The martingale $(M_t(F))_{t \geq 0}$ is unique up to indistinguishability.
\end{proposition}
\begin{proof}
Let $t \geq 0$, $\Gamma_t$ be a nonnegative, $\mathcal{F}_t$-measurable functional such that:
$$\mathbb{P} [\Gamma_t]= 0,$$
and let $f$ be an integrable, strictly positive function from $\mathbb{R}_+$ to $\mathbb{R}_+$.
 One has:
$$\mathbb{P} [M_t^f \, \Gamma_t] = 0,$$
and by Proposition \ref{f},
$$\mathcal{Q} [f(A_{\infty}) \, \Gamma_t] = 0.$$
Since $f$ is supposed to be strictly positive, one deduces that
$$\mathcal{Q} [\Gamma_t] = 0,$$
and finally,
$$\mathcal{Q} [F\, \Gamma_t]=0.$$
Therefore, the restriction of the finite measure $\mathcal{M}^F$ to $\mathcal{F}_t$ is absolutely continuous
with respect to $\mathbb{P}$, and there exists a nonnegative, $\mathcal{F}_t$-measurable random
 variable $M^{(0)}_t$ such that
for all $\mathcal{F}_t$-measurable, bounded variables $\Gamma_t$:
$$\mathcal{M}^F[\Gamma_t]= \mathbb{P} [M^{(0)}_t \, \Gamma_t].$$
This equality, available for all $t \geq 0$, implies that $(M^{(0)}_t)_{t \geq 0}$ is a $\mathbb{P}$-martingale.
Since the underlying probability space satisfies the natural conditions, $(M^{(0)}_t)_{t \geq 0}$ admits
a c\`adl\`ag modification $(M_t)_{t \geq 0}$, and one has the equality:
$$\mathcal{M}^F [\Gamma_t]= \mathbb{P} [M_t \, \Gamma_t].$$
By the monotone class theorem this determines uniquely the measure $\mathcal{M}^F$. Moreover, if $(M'_t)_{t \geq 0}$
is a c\`adl\`ag martingale satisfying:
$$\mathcal{M}^F [\Gamma_t]= \mathbb{P} [M'_t \, \Gamma_t],$$
then for all $t \geq 0$, $M_t = M'_t$ almost surely, and since $M$ and $M'$ are c\`adl\`ag, they are 
indistinguishable.
\end{proof}
\noindent
By Proposition \ref{MtF}, one can define a particular family of nonnegative, c\`adl\`ag $\mathbb{P}$-martingales:
the martingales of the form $(M_t(F))_{t \geq 0}$, where $F$ is a nonnegative, $\mathcal{Q}$-integrable 
functional $F$. By construction, these martingales correspond to the local densities, with respect to
 $\mathbb{P}$, of the finite measures which are absolutely continuous with respect to $\mathcal{Q}$.
The situation is similar to the case of nonnegative, uniformly integrable martingales, which are the local
 densities of the finite measures, absolutely continuous with respect to $\mathbb{P}$. 
Proposition \ref{MtF} does not give any explicit formula for the martingale $M_t(F)$.
However, from Proposition \ref{f}, one deduces immediately the following result:
\begin{corollary}
Under the assumptions of Theorem \ref{all}, for all integrable functions $f$ from $\mathbb{R}_+$ to 
$\mathbb{R}_+$, $f(A_{\infty})$ is integrable with respect to $\mathcal{Q}$, and the martingale 
$\big(M_t(f(A_{\infty})) \big)_{t \geq 0}$ is indistinguishable with the martingle $(M_t^f)_{t \geq 0}$
 defined in Proposition \ref{f}. 
\end{corollary}
\noindent
\begin{remark} 
Let $f$ be an integrable function from $\mathbb{R}_+$ to $\mathbb{R}_+$. The martingale
$$\left( \mathbb{P} [G(A_{\infty})| \mathcal{F}_t] \right)_{t \geq 0}$$ admits a c\`adl\`ag version.
If it is denoted by $(G_t)_{t \geq 0}$, one has:
$$M_t(F) = G(A_t) - G_t + Y_t,$$
where $(Y_t)_{t \geq 0}$ is a c\`adl\`ag modification of $(f(A_t)X_t)_{t \geq 0}$, which then exists
 for any choice of $f$ (recall that $G(A_t)$ is continuous with respect
to $t$). If $f$ is bounded, one easily proves that $f(A_t)X_t$ is c\`adl\`ag with respect to $t$: in this 
case, $(Y_t)_{t \geq 0}$ is indistinguishable with $(f(A_t)X_t)_{t \geq 0}$. However, for unbounded $f$, 
the existence of $Y$ is not trivial. 
\end{remark}
\noindent
Another case for which one can give a simple expression for the martingale $(M_t(F))_{t \geq 0}$ is
the case where $(X_t)_{t \geq 0}$ is of class (D). More precisely, one has the 
following result:
\begin{proposition} \label{uniformlyintegrable}
Let us suppose that the assumptions of Theorem \ref{all} are satisfied, and that the process
 $(X_t)_{t \geq 0}$ is
of class (D). Then $X_t$ tends a.s. to a limit $X_{\infty}$ when $t$ goes to infinity, and 
the measure $\mathcal{Q}$ is absolutely continuous with respect to $\mathbb{P}$, with density $X_{\infty}$. 
Moreover, a nonnegative measurable functional $F$ is integrable with respect to $\mathcal{Q}$ if and only if $F X_{\infty}$ is
 integrable with respect to $\mathbb{P}$, and in this case, $(M_t(F))_{t \geq 0}$ is a c\`adl\`ag version 
(unique up to indistinguishability) of the 
conditional expectation $(\mathbb{P} [F X_{\infty} | \mathcal{F}_t])_{t \geq 0}$. In particaular, it 
is uniformly integrable, and it converges a.s. and  in 
$L^1$ to $F X_{\infty}$ when $t$ goes to infinity. 
\end{proposition}
\begin{proof}
The equality $\mathcal{Q} = X_{\infty}. \mathbb{P}$ is  contained in \cite{NN1}, \cite{AMY} and
\cite{CNP}. Let us shortly reprove it here. Since $(X_t)_{t \geq 0}$ is of class (D), the 
expectation of $A_t$ is bounded, and then $A_{\infty}$ is
integrable, which implies that $(N_t)_{t \geq 0}$ is a uniformly integrable, c\`adl\`ag martingale. 
It admits an a.s. limit $N_{\infty}$ for $t$ going to infinity, and then $X_{\infty}$ is well-defined. 
Moreover, if $d_t$ is the infimum of $u > t$, such that $X_u =0$, by the version of the d\'ebut theorem given
 in \cite{NN2}, $d_t$ is a stopping time. Moreover, $d_t = \infty$ iff $g \leq t$, and by right-continuity of 
$X$, $X_{d_t}= 0$ for $d_t < \infty$. One deduces:
$$\mathbb{P} [X_{\infty} \mathds{1}_{g \leq t} | \mathcal{F}_t] = 
\mathbb{P} [X_{\infty} \mathds{1}_{d_t = \infty}| \mathcal{F}_t] = 
\mathbb{P} [X_{d_t} |\mathcal{F}_t].$$
Now, since $(N_t)_{t \geq 0}$ is a uniformly integrable, c\`adl\`ag martingale,
$$\mathbb{P} [X_{\infty} \mathds{1}_{g \leq t} |\mathcal{F}_t] 
= \mathbb{P} [N_{d_t} + A_{d_t} |\mathcal{F}_t] = N_t + A_t = X_t,$$
or equivalently, for all bounded, $\mathcal{F}_t$-measurable functionals $\Gamma_t$:
$$\mathbb{P} [\Gamma_t \, X_{\infty} \, \mathds{1}_{g \leq t} ] = \mathbb{P} [\Gamma_t \, X_t].$$
Moreover, under $X_{\infty}. \mathbb{P}$, $X_{\infty} >0$ almost everywhere and then $g < \infty$. 
One deduces that $X_{\infty}. \mathbb{P}$ is equal to $\mathcal{Q}$. 
For any nonnegative functional $F$, it is then trivial that $F$ is integrable with respect to $\mathcal{Q}$ iff
$FX_{\infty}$ is integrable with respect to $\mathbb{P}$, in this case, the finite measure $\mathcal{M}^F$ has 
density $F X_{\infty}$ with respect to $\mathbb{P}$. By taking the restriction to $\mathcal{F}_t$, one
deduces that the martingale $(M_t (F))_{t \geq 0}$ is a c\`adl\`ag version of the conditional expectation of 
$FX_{\infty}$. 
\end{proof}
\noindent
It is also possible to describe explicitly $\mathcal{Q}$ and $(M_t(F))_{t \geq 0}$ if $(X_t)_{t \geq 0}$ 
is a strictly positive martingale:
\begin{proposition} \label{mart}
Let us suppose that the assumptions of Theorem \ref{all} are satisfied, and that $\mathbb{P}$-almost 
surely, $(X_t)_{t \geq 0}$ 
does not vanish: in particular, $(A_t)_{t \geq 0}$ is indistinguishable from zero, and $(X_t)_{t \geq 0}$ is 
a martingale.  Then, $\mathcal{Q}$ is finite and it is the unique measure, such that 
for $t \geq 0$, the restriction of $\mathcal{Q}$ to $\mathcal{F}_t$ has density $X_t$ with respect to 
the restriction of $\mathbb{P}$ to $\mathcal{F}_t$. Moreover, for any nonnegative, $\mathcal{Q}$-integrable
functional $F$, the martingale $(M_t(F))_{t \geq 0}$ 
can be given by:
 $$M_t (F) = X_t \, \widetilde{\mathcal{Q}} [F \, | \mathcal{F}_t],$$ 
where $\widetilde{\mathcal{Q}} [F \, | \mathcal{F}_t]$ is a c\`adl\`ag version of the conditional expectation
of $F$ given $\mathcal{F}_t$, with respect to the probability measure $\widetilde{\mathcal{Q}}$ obtained 
by dividing $\mathcal{Q}$ by its total mass (different from zero). In particular, the 
functional identically equal to one is $\mathcal{Q}$-integrable and $(M_t(1))_{t \geq 0}$ is indistinguishable
 from $(X_t)_{t \geq 0}$.
\end{proposition}
\begin{proof}
Let $T_0 := \inf \{t \geq 0, X_t = 0\}$. By the d\'ebut theorem (under natural conditions), 
$T_0$ is an $(\mathcal{F}_t)_{t \geq 0}$-stopping time. By assumption, for all $t \geq 0$, 
the event $\{T_0 > t\}$, which is in $\mathcal{F}_t$, holds
 $\mathbb{P}$-almost surely. Now, by the construction of $\mathcal{Q}$ given in \cite{NN1} and 
described above, $\mathcal{Q}$ is absolutely
continuous with respect to a finite measure, which is locally absolutely continuous with respect to 
$\mathbb{P}$. One deduces that for all $t \geq 0$, the event $\{T_0 > t\}$ holds $\mathcal{Q}$-almost
 everywhere. Hence, $\mathcal{Q}$-almost everywhere, $T_0$ is infinite and $(X_t)_{t \geq 0}$ does not vanish, 
which implies
\begin{equation} 
\mathcal{Q} [\Gamma_t] = \mathbb{P} [\Gamma_t X_t]. \label{qq}
\end{equation}
By the monotone class theorem, $\mathcal{Q}$ is the unique measure satisfying (\ref{qq}): it is finite since 
$X_0$ is integrable, its total mass is different from zero since $X_0 > 0$. Hence, $\widetilde{\mathcal{Q}}$ 
is well-defined.
Moreover, if $F$ is integrable with respect to $\mathcal{Q}$, it is also
integrable with respect to $\widetilde{\mathcal{Q}}$, and the $\widetilde{\mathcal{Q}}$-martingale:
$$\left(\widetilde{\mathcal{Q}} [F \, | \mathcal{F}_t]\right)_{t \geq 0}$$
is well-defined and admits a c\`adl\`ag version $(Y_t)_{t \geq 0}$. Indeed, $(\Omega, \mathcal{F},
 (\mathcal{F}_t)_{t \geq 0},
\mathbb{P})$ satisfies the natural conditions, and then it is also the case for 
$(\Omega, \mathcal{F}, (\mathcal{F}_t)_{t \geq 0}, \widetilde{\mathcal{Q}} )$, since for all $t \geq 0$, 
the restriction of $\widetilde{\mathcal{Q}}$ to $\mathcal{F}_t$ is equivalent to 
the restriction of $\mathbb{P}$ (recall that $X_t > 0$, $\mathbb{P}$-almost surely). 
 Therefore, for all bounded, $\mathcal{F}_t$-measurable functionals $\Gamma_t$:
\begin{align*}
\mathcal{Q} [ F \, \Gamma_t] & = \mathcal{Q} [1] \, \widetilde{\mathcal{Q}} [F \, \Gamma_t]
\\ & = \mathcal{Q} [1] \, \widetilde{\mathcal{Q}} \left[\Gamma_t \,  \widetilde{\mathcal{Q}}
 [F \, | \mathcal{F}_t] \right] \\ & = \mathcal{Q}\left[\Gamma_t Y_t \right] \\ & =
 \mathbb{P} \left[ \Gamma_t X_t  Y_t \right]
\end{align*}
\noindent
Now, one has:
$$\mathcal{Q} [F \, \Gamma_t] = \mathcal{M}^F [\Gamma_t] = \mathbb{P} [\Gamma_t M_t(F)],$$
with the notation of Proposition \ref{MtF}. Hence, $(M_t(F))_{t \geq 0}$ is a modification 
of $(X_tY_t)_{t \geq 0}$, and since these two processes are c\`adl\`ag, they are indistinguishable. 
Moreover, the functional equal to one is $\mathcal{Q}$-integrable, since $\mathcal{Q}$ is finite. In
this case, one can take $Y_t=1$ for all $t \geq 0$, and $(M_t(F))_{t \geq 0}$ is 
indistinguishable from $(X_t)_{t \geq 0}$. 
\end{proof}
\noindent
After giving these simple examples for which one can explicitly compute $\mathcal{Q}$ and $M_t(F)$, it is
 natural to ask what happens in a more general situation. In Section \ref{sec::3}, we study the case where 
$A_{\infty} = \infty$, $\mathbb{P}$-almost surely (this case occurs, in particular, when $(X_t)_{t \geq 0}$
is a reflected Brownian motion). Unfortunately, we are not able to give explicit expressions for
the martingales of the form $(M_t(F))_{t \geq 0}$ in this case, but we 
obtain some information about their behaviour when $t$ goes to infinity. 
\section{The case $A_{\infty} = \infty$} \label{sec::3}
As it was proved in \cite{NN1}, the measure $\mathcal{Q}$ is infinite if one supposes that $A_{\infty} = \infty$,
$\mathbb{P}$-almost surely. More precisely, the image of $\mathcal{Q}$ by the functional $A_{\infty}$ is
the infinite measure: $\mathbb{P}[X_0] .\delta_0 + \lambda$, where $\lambda$ is Lebesgue measure 
on $\mathbb{R}_+$. Moreover, again for $A_{\infty} = \infty$, the martingale $(M_t(F))_{t \geq 0}$
tends $\mathbb{P}$-almost surely to zero for any $\mathcal{Q}$-integrable functional $F$. In particular, 
it cannot be uniformly integrable, except for $F=0$, $\mathcal{Q}$-almost everywhere. 
More precisely, one has the slightly more general result: 
\begin{proposition}
Let us suppose that the assumptions of Theorem \ref{all} are satisfied. Then on the set $\{A_{\infty} = 
\infty\}$, the martingale $M_t(F)$ tends $\mathbb{P}$-almost surely to zero when $t$ goes to infinity. 
\end{proposition}
\begin{proof}
Let us use the notation of Proposition \ref{MtF}. 
For all $u >0$, $v \geq t > 0$:
$$\mathcal{M}^F [A_t > u] = \mathbb{P} [M_v (F) \, \mathds{1}_{A_t > u}].$$
Moreover, $\mathbb{P}$-almost surely:
$$M_v (F) \, \mathds{1}_{A_t>u} \, \underset{v \rightarrow \infty}{\longrightarrow} \, 
M_{\infty} (F) \, \mathds{1}_{A_t > u},$$
where $M_{\infty} (F)$ is the a.s. limit of $M_t(F)$ for $t$ going to infinity. 
By Fatou's lemma, one deduces:
$$\mathbb{P} [M_{\infty}(F) \, \mathds{1}_{A_t > u}] \leq \mathcal{M}^F [A_t > u]
\leq \mathcal{M}^F [A_{\infty} > u].$$
Now, $M_{\infty}(F) \, \mathds{1}_{A_{\infty} > u }$ is the almost sure limit of
$M_{\infty} (F) \, \mathds{1}_{A_t > u}$. Since $M_{\infty} (F)$ is integrable by Fatou's lemma, 
one has, by dominated convergence:
$$\mathbb{P} [M_{\infty}(F) \, \mathds{1}_{A_{\infty} > u }] \leq \mathcal{M}^F [A_{\infty} > u].$$
By taking $u$ going to infinity, we are done, since $A_{\infty}$ is finite 
$\mathcal{M}^F$-almost everywhere.
\end{proof}
\noindent
Once the behaviour of $(M_t(F))_{t \geq 0}$ under $\mathbb{P}$ is known, it is natural to ask
what happens under $\mathcal{Q}$. The following result implies that the behaviour of 
$(M_t(F))_{t \geq 0}$ is not the same. Moreover, it gives some information about the behaviour of 
$(X_t)_{t \geq 0}$ under $\mathcal{Q}$: 
\begin{proposition} \label{ntmq}
Let us suppose that the assumptions of Theorem \ref{all} are satisfied, and that
 $A_{\infty} = \infty$, $\mathbb{P}$-almost surely. Then $\mathcal{Q}$-almost everywhere, $X_t$ tends 
to infinity with $t$, and
$$\frac{M_t(F)}{X_t} \underset{t \rightarrow \infty}{\longrightarrow} F$$
for all nonnegative, $\mathcal{Q}$-integrable functionals $F$. 
\end{proposition}
\noindent
\begin{remark}
As we have seen in Proposition \ref{MtF}, two versions of $(M_t(F))_{t \geq 0}$ are indistinguishable
with respect to $\mathbb{P}$. Since $\mathcal{Q}$ is absolutely continuous with respect to a finite measure
which is locally absolutely continuous with respect to $\mathbb{P}$, the two versions are also 
indistinguishable with respect to $\mathcal{Q}$. Hence, $(M_t(F))_{t \geq 0}$ can be considered to be 
well-defined for all the problems concerning its behaviour under the measure $\mathcal{Q}$. 
 \end{remark}
\begin{proof}
The functional $H := e^{-A_{\infty}}$ is $\mathcal{Q}$-integrable and one has:
$$M_t(H) = e^{-A_t} (1+X_t)$$
(recall that $\mathbb{P}$-almost surely, $e^{-A_{\infty}} = 0$, since $A_{\infty} = \infty$).
One deduces, for all bounded, $\mathcal{F}_t$-measurable random variables $\Gamma_t$:
$$\mathcal{M}^H [\Gamma_t] = \mathbb{P} \left[e^{-A_t} (1+X_t) \, \Gamma_t \right].$$
This implies:
$$\mathbb{P} [\Gamma_t] = \mathcal{M}^H [Y_t \Gamma_t],$$
where 
$$Y_t = \frac{e^{A_t}} {1+ X_t}.$$
Now, $(Y_t)_{t \geq 0}$ is a nonnegative, c\`adl\`ag martingale, with respect to the probability measure 
$\widetilde{\mathcal{M}}^H := \mathcal{M}^H/ \mathcal{M}^H (1)$, and then, converges $\mathcal{M}^H$-almost everywhere to 
a limit random variable $Y_{\infty}$. Now, since for all $u >0$, $v \geq t >0$:
$$\mathbb{P} [A_t \leq u] = \mathcal{M}^H [Y_v \, \mathds{1}_{A_t \leq u} ],$$
one has, by taking $v \rightarrow \infty$ and by using Fatou's lemma:
$$\mathcal{M}^H [Y_{\infty} \, \mathds{1}_{A_t \leq u} ] \leq \mathbb{P} [A_t \leq u],$$
which implies
$$\mathcal{M}^H [Y_{\infty} \, \mathds{1}_{A_{\infty} \leq u}] \leq 
\mathbb{P} [A_t \leq u].$$
Now, since $A_{\infty} = \infty$, $\mathbb{P}$-almost surely, one has
$$\mathbb{P} [A_t \leq u]  \underset{t \rightarrow \infty}{\longrightarrow} 0.$$
Hence, 
$$\mathcal{M}^H [Y_{\infty} \, \mathds{1}_{A_{\infty} \leq u}] = 0,$$
and finally (by taking $u$ going to infinity):
$$\mathcal{M}^H [Y_{\infty} \, \mathds{1}_{A_{\infty} < \infty}] = 0.$$
Since $A_{\infty} < \infty$, $\mathcal{Q}$-almost everywhere, $Y_{\infty} = 0$, $\mathcal{M}^H$-almost
everywhere, which implies that $X_t$ tends to infinity with $t$. 
On the other hand, for all nonnegative, integrable 
functionals $F$, and for all bounded, $\mathcal{F}_t$-measurable functionals $\Gamma_t$, one has:
\begin{align*}
\mathcal{M}^H \left[\Gamma_t \, \frac{M_t(F)}{M_t(H)} \right] 
= \mathcal{Q} \left[\Gamma_t \, H \, \frac{M_t(F)}{M_t(H)} \right] & 
= \mathbb{P} \left[\Gamma_t \, M_t(H) \, \frac{M_t(F)}{M_t(H)} \right]
 = \mathbb{P} [\Gamma_t \, M_t(F) ] \\ & = \mathcal{Q} [\Gamma_t \, F] \\
& = \mathcal{M}^H  \left[ \Gamma_t \, \frac{F}{H} \right] 
 = \mathcal{M}^H \left[ \Gamma_t \, \widetilde{\mathcal{M}}^H \left[ \frac{F}{H} \, | \mathcal{F}_t \right] \right]. 
\end{align*}
\noindent
Note that all the equalities above are meaningful since $M_t(H)$ and $H$ never vanish. Therefore, for all 
$t \geq 0$, one has almost surely:
$$\frac{M_t(F)}{M_t(H)} = \widetilde{\mathcal{M}}^H \left[ \frac{F}{H} \, | \mathcal{F}_t \right],$$
which implies that 
$$\frac{M_t(F)}{M_t(H)} \underset{t \rightarrow \infty}{\longrightarrow} \frac{F}{H},$$
 $\widetilde{\mathcal{M}}^H$-almost surely, and then, $\mathcal{Q}$-almost everywhere. 
Now, since $X_t \rightarrow \infty$, $\mathcal{Q}$-almost everywhere, $X_t >0$ for $t$ large enough and:
$$\frac{M_t(H)}{X_t} = e^{-A_t} \left(1 + \frac{1}{X_t}\right) \underset{t \rightarrow \infty}{\longrightarrow}
e^{-A_{\infty}}.$$
One deduces:
$$\frac{M_t(F)}{X_t} = \frac{M_t(F)}{M_t(H)} \, \frac{M_t(H)}{X_t} 
\underset{t \rightarrow \infty}{\longrightarrow} \, \frac{F}{H} \, e^{-A_{\infty}} = F.$$
\end{proof}
\noindent
In the case where $(X_t)_{t \geq 0}$ is a reflected Brownian motion, Proposition \ref{ntmq} is essentially proved
in \cite{NRY} and when $X_t$ is a symmetric $\alpha$-stable process of index $\alpha\in(1,2)$, it is proved in \cite{YYY}. In the  particular case of the reflected Brownian motion, the measure $\mathcal{Q}$ is
 strongly related to the last passage time at any level and not only at zero. 
This relation can be generalized as follows:
\begin{proposition} \label{xta}
Let us suppose that the assumptions of Theorem \ref{all} are satisfied, and that the submartingale 
$(X_t)_{t \geq 0}$ has only positive jumps and that $A_{\infty} = \infty$ almost surely 
under $\mathbb{P}$. For $a \geq 0$, let $g^{[a]}$ be the last hitting time of the 
interval $[0,a]$:
$$g^{[a]} = \sup \{t \geq 0, X_t \leq a \}.$$
Then  for all $t \geq 0$, and for all
$\mathcal{F}_t$-measurable, bounded variables $\Gamma_t$, the measure $\mathcal{Q}$ satisfies
\begin{equation}\mathcal{Q} \left[ \Gamma_t \, \mathds{1}_{g^{[a]} \leq t} \right] = \mathbb{P} 
\left[\Gamma_t (X_t-a)_+ \right]. \label{Qa}
\end{equation}
\noindent
Moreover, $((X_t-a)_+)_{t \geq 0}$ is a submartingale of class $(\Sigma)$ and the $\sigma$-finite measure 
obtained by applying 
Theorem \ref{all} to it is equal to $\mathcal{Q}$. 
\end{proposition}
\begin{proof}
Let:
$$d^{[a]}_t := \inf \{v > t, X_v \leq a \}. $$
By the d\'ebut theorem (for natural conditions), $d^{[a]}_t$ is a stopping time.
Now, for all $u > t$:
$$\mathcal{Q} \left[\Gamma_t \, \mathds{1}_{g \leq u, d^{[a]}_t > u} \right] = 
\mathbb{P} \left[\Gamma_t \, \mathds{1}_{d^{[a]}_t > u} X_u \right].$$
One deduces, by using the decomposition of the submartingale $(X_t)_{t \geq 0}$, and 
by applying martingale property to $(N_t)_{t \geq 0}$:
\begin{align*}
\mathcal{Q} \left[\Gamma_t \, \mathds{1}_{g \leq u, d^{[a]}_t > u} \right] &  = 
\mathbb{P} \left[\Gamma_t \, \mathds{1}_{d^{[a]}_t > u} (N_u + A_u) \right] \\ 
&  = \mathbb{P} \left[\Gamma_t \, \mathds{1}_{d^{[a]}_t > u} (N_u + A_t) \right] \\
& = \mathbb{P} \left[\Gamma_t \, \mathds{1}_{d^{[a]}_t > u} A_t \right] \\ & \; 
+ \mathbb{P} \left[\Gamma_t \, N_u \right] - \mathbb{P} \left[\Gamma_t \, \mathds{1}_{d^{[a]}_t \leq u}  N_u
\right] \\ & =  \mathbb{P} \left[\Gamma_t \, \mathds{1}_{d^{[a]}_t > u} A_t \right] \\ & \;  
+ \mathbb{P} \left[\Gamma_t \, N_t \right]  - \mathbb{P} \left[\Gamma_t \, \mathds{1}_{d^{[a]}_t \leq u}
  N_{d^{[a]}_t} \right] \\ 
& = \mathbb{P} \left[\Gamma_t \, \mathds{1}_{d^{[a]}_t > u} A_t \right] 
+ \mathbb{P} \left[\Gamma_t \, N_t \right]  \\ \; & -  \mathbb{P} \left[\Gamma_t \, \mathds{1}_{d^{[a]}_t \leq u}
  X_{d^{[a]}_t} \right] +  \mathbb{P} \left[\Gamma_t \, \mathds{1}_{d^{[a]}_t \leq u}
  A_t \right] \\ & = \mathbb{P} \left[\Gamma_t X_t \right] 
- \mathbb{P} \left[\Gamma_t \, \mathds{1}_{d^{[a]}_t \leq u}
  X_{d^{[a]}_t} \right].
\end{align*}
\noindent
Now, by right continuity, $d^{[a]}_t = t$ if $X_t < a$, and since $X$ has only positive jumps, for $X_t \geq a$ and $d^{[a]}_t < \infty$,
$X_{d^{[a]}_t} = a$. One deduces that
$$\mathcal{Q} \left[\Gamma_t \, \mathds{1}_{g \leq u, d^{[a]}_t > u} \right] 
= \mathbb{P} \left[\Gamma_t X_t \right] - \mathbb{P} \left[\Gamma_t \, \mathds{1}_{d^{[a]}_t \leq u}
  (X_t \wedge a) \right].$$
When $u$ tends to infinity, the event $\{g \leq u, d^{[a]}_t  > u \}$ tends to the  
event $\{g^{[a]} \leq t\}$.
Moreover, the event $\{d^{[a]}_t \leq u\}$ tends to 
$\{d^{[a]}_t < \infty\}$, which is almost sure under $\mathbb{P}$, since $A_{\infty} = \infty$. 
One deduces:
$$\mathcal{Q} \left[\Gamma_t \mathds{1}_{g^{[a]} \leq t} \right]  = 
\mathbb{P} \left[\Gamma_t X_t \right] - \mathbb{P} \left[\Gamma_t (X_t \wedge a) \right]
= \mathbb{P} \left[\Gamma_t (X_t-a)_+ \right].$$
Now, from  Lemma 2.1 in \cite{CNP}, $((X_t-a)_+)_{t \geq 0}$ is also nonnegative submartingale of class $(\Sigma)$. The supremum of its hitting times
of zero is $g^{[a]}$. The formula \eqref{Qa} and the fact that $\{g^{[a]} < \infty \}$ holds
$\mathcal{Q}$-almost everywhere (recall that $X_t \rightarrow \infty$ when $t \rightarrow \infty$, 
since $A_{\infty} = \infty$, $\mathbb{P}$-almost surely), imply that 
$\mathcal{Q}$ is the $\sigma$-finite measure obtained from the submartingale $((X_t-a)_+)_{t \geq 0}$. 
\end{proof}
\noindent
In their study of Brownian penalisations, Najnudel, Roynette and Yor (\cite{NRY}) introduce
a particular class of nonnegative processes which converge $\mathcal{Q}$-almost everywhere to 
a $\mathcal{Q}$-integrable functional. Let us state a similar definition in our general 
framework:
\begin{definition}
Let us suppose that the assumptions of Theorem \ref{all} are satisfied. We say that 
a process $(F_t)_{t \geq 0}$ belongs to the class (C) if it is nonnegative, uniformly bounded,
nonincreasing, c\`adl\`ag and adapted with 
respect to $(\mathcal{F}_t)_{t \geq 0}$, if there exists $a > 0$ such that for all 
$t \geq 0$, $F_t = F_{g^{[a]}}$ on the set $\{t \geq g^{[a]}\}$, and if its decreasing limit at infinity,
denoted $F_{\infty}$, is $\mathcal{Q}$-integrable.
\end{definition}
\noindent
For example, the process $(F_t)_{t \geq 0}$ given by $$F_t=\varphi(A_t),$$where $\varphi: \mathbb{R}_+\to\mathbb{R}_+$ is integrable and decreasing,  is in the class (C), as well as 
$$F_t :=  \exp \left( - \lambda A_t - \int_0^t q(X_s) ds \right),$$
where $\lambda  > 0$ and where $q$ is a measurable function from 
$\mathbb{R}_+$ to $\mathbb{R}_+$, with compact support.
When a process $(F_t)_{t \geq 0}$ is in the class (C), the following proposition gives the behaviour of 
$\mathbb{P} [F_tX_t]$ when $t$ goes to infinity.
\begin{proposition}\label{nardine2}
Let us suppose that the assumptions of Theorem \ref{all} are satisfied, and that $A_{\infty}= \infty$, 
$\mathbb{P}$-almost surely. Let $(F_t)_{t \geq 0}$ be a process in the class (C).
Then, if $F_g$ is integrable with respect to $\mathcal{Q}$, one has:
$$\mathbb{P} [F_tX_t ] \underset{t \rightarrow \infty}{\longrightarrow} \mathcal{Q} [F_{\infty}]. $$
\end{proposition}
\begin{proof}
It is sufficient to prove:
$$\mathcal{Q} [F_t \mathds{1}_{g \leq t} ] \underset{t \rightarrow \infty}
{\longrightarrow} \mathcal{Q} [F_{\infty}].$$
Now, since the set $\{g^{[a]}\leq t\}$ is included in the set $\{g \leq t\}$, one can write:
$$\mathcal{Q} [F_t \mathds{1}_{g \leq t}] = \mathcal{Q} [F_t \mathds{1}_{g^{[a]} \leq t} ] 
+ \mathcal{Q} [F_t \mathds{1}_{g \leq t < g^{[a]}}].$$
Moreover:
$$\mathcal{Q} [F_t \mathds{1}_{g^{[a]} \leq t}] = \mathcal{Q} [F_{\infty} \mathds{1}_{g^{[a]} \leq t}]
\underset{t \rightarrow \infty}{\longrightarrow} \mathcal{Q} [F_{\infty} \mathds{1}_{g^{[a]}< \infty}] = 
\mathcal{Q}[F_{\infty}].$$
The last equality is due to the fact that in the case where $A_{\infty} = \infty$, $\mathbb{P}$-almost surely, 
the process $(X_t)_{t \geq 0}$ tends $\mathcal{Q}$-almost everywhere to infinity with $t$. 
Hence, it is sufficient to prove that
$$\mathcal{Q} [F_t \mathds{1}_{g \leq t < g^{[a]}}] \underset{t \rightarrow \infty}{\longrightarrow} 
0.$$
Now, $F_t \mathds{1}_{g \leq t < g^{[a]}}$ is dominated by $F_g$, integrable with respect to $\mathcal{Q}$,
and tends to zero, $\mathcal{Q}$-almost surely, for $t$ going to infinity. By dominated convergence, we are done.
\end{proof}
\begin{remark}
Let $(X_t)_{t \geq 0}$ be the absolute value of the Wiener process and let  $F_t:=\exp(-\lambda L_t)$, where
 $L_t$ is the local time of $(X_t)_{t \geq 0}$ at level $0$. The process $(F_t)_{t \geq 0}$ is in the
 class (C) and it is known (see \cite{NN1}) that $L_\infty$ follows the Lebesgue measure on
 $\mathbb{R}_{+}$ under $\mathcal{Q}$. Consequently, 
$$ \mathbb{P}[\exp(-\lambda L_t) X_t] \underset{t \rightarrow \infty}{\longrightarrow} 1/\lambda,$$
 although 
$$\exp(-\lambda L_t) X_t \underset{t \rightarrow \infty}{\longrightarrow} 0,$$
 $\mathbb{P}$-almost surely. Of course, due to the general feature of our results, the same result holds if one replaces $X_t$ by $|Y_t|^{\alpha-1}$, where $Y$ is a symmetric $\alpha$-stable L\'evy process with index $\alpha\in(1,2)$, and $L_t$ would then stand for the local time of $Y$.
\end{remark}
\noindent
Here is another version of the same result (which does not involve the class (C)) and which is in fact more powerful and useful:
\begin{proposition}\label{nardine1}
Let us suppose that the assumptions of Theorem \ref{all} are satisfied and that $A_{\infty}= \infty$, 
$\mathbb{P}$-almost surely. Let $(F_t)_{t \geq 0}$ be a c\`adl\`ag,
 adapted, nonnegative process such that its limit $F_{\infty}$ exists $\mathcal{Q}$-almost 
everywhere. We suppose that there exists a $\mathcal{Q}$-integrable, nonnegative 
functional $H$, such that for all $t \geq 0$, one has:
$$F_t X_t \leq M_t(H)$$
$\mathbb{P}$-almost surely. Then, $F_{\infty}$ is $\mathcal{Q}$-integrable and:
$$\mathbb{P} [F_tX_t ] \underset{t \rightarrow \infty}{\longrightarrow} \mathcal{Q} [F_{\infty}]. $$
\end{proposition}
\begin{proof}
For all $t \geq 0$, one has $\mathcal{Q}$-almost everywhere
\begin{equation} \label{fxm}
F_t X_t \leq M_t(H).
\end{equation}
\noindent
Indeed, the event $\{F_t X_t > M_t(H)\}$ is $\mathcal{F}_t$-measurable and $\mathbb{P}$-negligible, and 
then, $\mathcal{Q}$-negligible. One deduces that $\mathbb{P}$-almost surely and $\mathcal{Q}$-almost
 everywhere, \eqref{fxm}
is satisfied for all rationals $t \geq 0$, and then for all $t \geq 0$, since $(F_t)_{t \geq 0}$, 
$(M_t(H))_{t \geq 0}$ and $(X_t)_{t \geq 0}$ are c\`adl\`ag.  
By adding $e^{-A_{\infty}}$ to $H$, one can now suppose that $H > 0$ and $M_t(H) >0$ for all $t$. 
Hence:
$$\mathbb{P} [F_tX_t] = \mathcal{M}^{H} \left[ \frac{F_tX_t}{M_t(H)} \right].$$
Now, one has, uniformly in $t$:
$$\frac{F_tX_t}{M_t(H)} \leq 1 + \infty . \mathds{1}_{\exists t \geq 0, F_t X_t > M_t(H) },$$
which is $\mathcal{M}^H$-integrable, since $\mathcal{M}^H$ is a finite measure and 
the event $\{\exists t \geq 0, F_t X_t > M_t(H) \}$ is $\mathcal{Q}$, and then $\mathcal{M}^H$-negligible. 
In particular:
$$\mathbb{P} [F_t X_t] \leq \mathcal{M}^H \left[1 + \infty . \mathds{1}_{\exists t \geq 0, F_t X_t > M_t(H) } 
\right] = \mathcal{Q} [H] < \infty.$$
Moreover, $\mathcal{Q}$-almost everywhere:
$$\frac{F_tX_t}{M_t(H)} \underset{t \rightarrow \infty}{\longrightarrow} \frac{F_{\infty}}{H}.$$
By dominated convergence:
$$\mathbb{P} [F_tX_t] \underset{t \rightarrow \infty}{\longrightarrow} \mathcal{M}^H 
\left[\frac{F_{\infty}}{H} \right] = \mathcal{Q} [F_{\infty}].$$
Since for all $t \geq 0$, $$\mathbb{P} [F_tX_t] \leq \mathcal{Q}[H],$$
one deduces that $$\mathcal{Q} [F_{\infty}] \leq \mathcal{Q} [H] < \infty.$$
\end{proof}
\noindent
We now illustrate how the above result can be used. Let $f:\mathbb{R}_+\to\mathbb{R}_+$ be an integrable function. From the identity (\ref{MTF2}) defining the
 martingale $(M_t^f)_{t \geq 0}$, and using the fact that $A_\infty=\infty$, $\mathbb{P}$-almost surely, we have that
$$f(A_t)X_t\leq M_t^f.$$ Consequently the above Proposition applies to the case $F_t=f(A_t)$, with
 $f:\mathbb{R}_+\to\mathbb{R}_+$ an integrable function.  It also obviously applies to any 
function $F_t$ which satisfying $F_t\leq f(A_t)$, for some integrable  $f:\mathbb{R}_+\to\mathbb{R}_+$. For 
instance, the result would apply to any $F_t=G_t f(A_t)$ where $G_t$ is a bounded 
c\`adl\`ag $\mathcal{F}_t$-measurable process and $f:\mathbb{R}_+\to\mathbb{R}_+$ is integrable; in particular if $F_t=f(A_t)\exp \left(- \int_0^t q(X_s) ds \right),$
where  $q$ is a measurable function from 
$\mathbb{R}_+$ to $\mathbb{R}_+$, then the above proposition applies.

We are now able to state two universal  penalisation results:
\begin{proposition}
Let us suppose that the assumptions of Theorem \ref{all} are satisfied, and that $A_{\infty}= \infty$, 
$\mathbb{P}$-almost surely. Let $(F_t)_{t \geq 0}$ be a process in the class (C)
such that $F_g$ is integrable with respect to $\mathcal{Q}$ and $F_{\infty}$ is not
$\mathcal{Q}$-almost everywhere equal to zero. Then, for $t$ sufficiently large,
$0<\mathbb{P}[F_tX_t]<\infty$, and one can define a measure $\mathbb{Q}_t$ by
$$\mathbb{Q}_t=\dfrac{F_t X_t}{\mathbb{P}[F_tX_t]}. \mathbb{P}.$$ Moreover, there exists a probability 
measure $\mathbb{Q}_{\infty}$ which can be considered as the weak limit of $\mathbb{Q}_t$ 
when $t$ goes to infinity, in the following sense: for all $s \geq 0$ and for all events 
$\Lambda_s \in \mathcal{F}_s$,
$$\mathbb{Q}_t [\Lambda_s] \underset{t \rightarrow \infty}{\longrightarrow} \mathbb{Q}_{\infty} [\Lambda_s].$$
The measure $\mathbb{Q}_{\infty}$ is absolutely continuous with respect to $\mathcal{Q}$:
$$\mathbb{Q}_\infty=\dfrac{F_\infty}{\mathcal{Q}[F_\infty]}.\mathcal{Q},$$
where $0 < \mathcal{Q}[F_{\infty}] < \infty$.
\end{proposition}
\begin{proof}
The integrability of $F_{\infty}$ under $\mathcal{Q}$ is an immediate consequence of the integrability 
of $F_g$, and $\mathcal{Q}[F_{\infty}] > 0$ because $F_{\infty}$ is not $\mathcal{Q}$-almost everywhere 
equal to zero. Moreover, for all $t \geq 0$, $F_t$ is uniformly bounded and $X_t$ is $\mathbb{P}$-integrable, 
which implies that $\mathbb{P} [F_tX_t]$ is finite. On the other hand, by Proposition \ref{nardine2}, 
\begin{equation}
\mathbb{P} [F_tX_t ] \underset{t \rightarrow 
\infty}{\longrightarrow} \mathcal{Q} [F_{\infty}]  >0,  \label{chu}
\end{equation} 
and then $\mathbb{P} [F_tX_t ] > 0$ for $t$ large enough. Now, one has, 
for $t >s$:
 $$\mathbb{P}[F_t \mathds{1}_{\Lambda_s} X_t ]=\mathcal{Q}[F_t \mathds{1}_{\Lambda_s} \mathds{1}_{g \leq t}],$$
 where, by the arguments in the proof of Proposition \ref{nardine2},
 $$\mathcal{Q} [ F_t \mathds{1}_{\Lambda_s} \mathds{1}_{g \leq t} ] \underset{t \rightarrow \infty}
{\longrightarrow} \mathcal{Q} [F_{\infty}\mathds{1}_{\Lambda_s}].$$Combining this with (\ref{chu})
 completes the proof of the proposition.
\end{proof}
\begin{proposition}
Let us suppose that the assumptions of Theorem \ref{all} are satisfied and that $A_{\infty}= \infty$, 
$\mathbb{P}$-almost surely. Let $(F_t)_{t \geq 0}$ be a c\`adl\`ag,
 adapted, nonnegative process such that its limit $F_{\infty}$ exists $\mathcal{Q}$-almost 
everywhere and is not $\mathcal{Q}$-almost everywhere equal to zero. We suppose that there exists
 a $\mathcal{Q}$-integrable, nonnegative functional $H$, such that for all $t \geq 0$, one has:
$$F_t X_t \leq M_t(H)$$
$\mathbb{P}$-almost surely. Then, for $t$ sufficiently large, 
$0<\mathbb{P}[F_tX_t]<\infty$ and one can define a measure $\mathbb{Q}_t$ by
$$\mathbb{Q}_t=\dfrac{F_t X_t}{\mathbb{P}[F_tX_t]}. \mathbb{P}.$$ Moreover, there exists a probability 
measure $\mathbb{Q}_{\infty}$ which can be considered as the weak limit of $\mathbb{Q}_t$ 
when $t$ goes to infinity, in the following sense: for all $s \geq 0$ and for all events 
$\Lambda_s \in \mathcal{F}_s$,
$$\mathbb{Q}_t [\Lambda_s] \underset{t \rightarrow \infty}{\longrightarrow}
 \mathbb{Q}_{\infty} [\Lambda_s].$$ The measure $\mathbb{Q}_{\infty}$ is 
absolutely continuous with $\mathcal{Q}$:
$$\mathbb{Q}_\infty=\dfrac{F_\infty}{\mathcal{Q}[F_\infty]}.\mathcal{Q},$$
where $0 < \mathcal{Q}[F_{\infty}] < \infty$. 
\end{proposition}
\begin{proof}
Since for all $t \geq 0$, $F_t X_t \leq M_t(H)$, $\mathbb{P}$-almost surely, one has
$$\mathbb{P} [F_tX_t] \leq \mathbb{P} [M_t(H)] = \mathcal{Q} [H] < \infty.$$
On the other hand, by Proposition \ref{nardine1}, 
\begin{equation}
\mathbb{P} [F_tX_t] \underset{t \rightarrow \infty}{\longrightarrow} \mathcal{Q}[F_{\infty}] \in (0,\infty), \label{chu2}
\end{equation}
which implies that $\mathbb{P} [F_tX_t]  >0$ for $t$ large enough. Moreover, by applying Proposition 
\ref{nardine1} to the family of functionals $(F_t \mathds{1}_{\Lambda_s} \mathds{1}_{t \geq s} )_{t \geq 0}$, 
one deduces:
$$\mathbb{P} [F_t \mathds{1}_{\Lambda_s} X_t] \underset{t \rightarrow \infty}{\longrightarrow} 
\mathcal{Q}[\mathds{1}_{\Lambda_s} F_{\infty}].$$
Combining this with (\ref{chu2}) completes the proof of the proposition.
\end{proof}
\begin{remark}
The results above give the behaviour of the quantity $\mathbb{P} [F_tX_t]$. In order to obtain penalisation
results which do not necessarily involve $X_t$, e.g. of the form $F_t=f(A_t)$, we need to find an equivalent for $\mathbb{P}[F_t]$. Unfortunately, we are not able to give such an 
estimate in the general case, however, if $(X_t)_{t \geq 0}$ is a diffusion satisfying some technical conditions, 
this problem is solved in the companion paper \cite{NN4}, and we deduce a penalisation theorem, generalizing 
 results given in \cite{NRY}. 
 \end{remark}
\section{A new decomposition of nonnegative supermartingales} \label{sec::4}
The following proposition gives a general decomposition of any nonnegative, c\`adl\`ag supermartingale, involving 
a uniformly martingale and a martingale of the form $(M_t(F))_{t \geq 0}$. This decomposition generalizes
a result obtained in \cite{NRY} (Theorem 1.2.5).
\begin{proposition} \label{decomposition}
Let us suppose that the assumptions of Theorem \ref{all} are satisfied, and that
 $A_{\infty} = \infty$, $\mathbb{P}$-almost surely. Let $Z$ be a nonnegative, c\`adl\`ag
 $\mathbb{P}$-supermartingale. 
We denote by $Z_{\infty}$ the $\mathbb{P}$-almost sure limit of $Z_t$ when $t$ goes to infinity. 
Then, $\mathcal{Q}$-almost everywhere, 
the quotient $Z_t/X_t$ is well-defined for $t$ large enough and converges, when 
$t$ goes to infinity, to a limit $z_{\infty}$, integrable
 with respect to $\mathcal{Q}$,  and $(Z_t)_{t \geq 0}$ decomposes
 as
$$\left(Z_t = M_t (z_{\infty}) + \mathbb{P} [Z_{\infty} |\mathcal{F}_t] + \xi_t \right)_{t \geq 0},$$
where $(\mathbb{P} [Z_{\infty} |\mathcal{F}_t])_{t \geq 0}$ denotes a c\`adl\`ag version of the 
conditional expectation of $Z_{\infty}$ with respect to $\mathcal{F}_t$, and $(\xi_t)_{t \geq 0}$ is 
a nonnegative, c\`adl\`ag $\mathbb{P}$-supermartingale, such that:
\begin{itemize}
\item $Z_{\infty} \in L^1_+ (\mathcal{F}, \mathbb{P})$, hence $\mathbb{P} [Z_{\infty} | \mathcal{F}_t]$ 
converges $\mathbb{P}$-almost surely and in $L^1 (\mathcal{F}, \mathbb{P})$ towards $Z_{\infty}$;
\item $\frac{\mathbb{P} [Z_{\infty} | \mathcal{F}_t] + \xi_t}{X_t} \underset{t \rightarrow \infty}
{\longrightarrow} 0$, $\mathcal{Q}$-almost everywhere;
\item $M_t(z_{\infty}) + \xi_t \underset{t \rightarrow \infty}{\longrightarrow} 0$, $\mathbb{P}$-almost surely.  
\end{itemize}
\noindent
Moreover, the decomposition is unique in the following sense: let $z'_{\infty}$ be a $\mathcal{Q}$-integrable, 
nonnegative functional, $Z'_{\infty}$ a $\mathbb{P}$-integrable, nonnegative random variable,
 $(\xi'_t)_{t \geq 0}$ a c\`adl\`ag, nonnegative $\mathbb{P}$-supermartingale, and let us suppose that
 for all $t\geq 0$,
$$Z_t = M_t (z'_{\infty}) + \mathbb{P} [Z'_{\infty} |\mathcal{F}_t] + \xi'_t.$$
Under these assumptions, if for $t$ going to infinity, $\xi'_t$ tends $\mathbb{P}$-almost surely to zero and
 $\xi'_t /X_t$ tends $\mathcal{Q}$-almost everywhere to zero, then $z'_{\infty} = z_{\infty}$, $\mathcal{Q}$-almost
everywhere, 
$Z'_{\infty} = Z_{\infty}$, $\mathbb{P}$-almost surely, and $\xi'$ is $\mathbb{P}$-indistinguishable with $\xi$. 
\end{proposition}
\begin{proof}
Let $H:= e^{-A_{\infty}}$. Since $Z$ is a c\`adl\`ag $\mathbb{P}$-supermartingale, it is 
easy to deduce that 
$$\left( \frac{Z_t}{M_t(H)} \right)_{t \geq 0}$$
is a c\`adl\`ag supermartingale with respect to $\widetilde{\mathcal{M}}^H := \mathcal{M}^H/ \mathcal{M}^H (1)$. Hence, 
it converges $\widetilde{\mathcal{M}}^H$-almost surely to a limit $\zeta$. Since $M_t(H)/X_t$ 
 converges $\widetilde{\mathcal{M}}^H$-a.s. to $H$, $Z_t/X_t$  converges $\widetilde{\mathcal{M}}^H$-a.s., 
and then $\mathcal{Q}$-almost everywhere, to $z_{\infty} = \zeta \, H$. 
Moreover, since $\zeta$ is the $\widetilde{\mathcal{M}}^H$-a.s. limit of the 
$\widetilde{\mathcal{M}}^H$-supermartingale $(Z_t/ M_t(H))_{t \geq 0}$, one has:
$$\mathcal{Q} [z_{\infty}] = \mathcal{M}^H [\zeta] \leq \mathcal{M}^H [Z_0/M_0(H)] < \infty.$$
Since $z_{\infty}$ is $\mathcal{Q}$-integrable, $(M_t(z_{\infty}))_{t \geq 0}$ is well-defined. 
Now, for all nonnegative, $\mathcal{F}_t$-measurable functionals $\Gamma_t$:
\begin{align*}
\mathcal{Q} [\Gamma_t z_{\infty}] & = \mathcal{Q} \left[ \Gamma_t \, \underset{u \rightarrow \infty}
{\lim} \, \frac{Z_u}{X_u}  \right] \\ & = \mathcal{Q} \left[ \Gamma_t \, \underset{u \rightarrow \infty}
{\lim} \, \frac{Z_u}{X_u}  \, \mathds{1}_{g \leq u} \right] \\ &
\leq \underset{u \rightarrow \infty}{\lim \inf} \, \mathcal{Q} \left[\Gamma_t \, \frac{Z_u}{X_u} 
\, \mathds{1}_{g \leq u} \right] \\ & = \underset{u \rightarrow \infty}{\lim \inf} 
\, \mathbb{P} \left[ \Gamma_t \, \frac{Z_u}{X_u} \, X_u \right] \\ & \leq 
\underset{u \rightarrow \infty}{\lim \inf} \, \mathbb{P} [\Gamma_t Z_u] \leq \mathbb{P} [\Gamma_t Z_t].
\end{align*}
\noindent
One deduces that for all $t \geq 0$, $M_t(z_{\infty}) \leq Z_t$, $\mathbb{P}$-a.s., which implies 
that $ \left(M_t(z_{\infty}) \wedge Z_t \right)_{t \geq 0}$ is a c\`adl\`ag and adapted modification of 
$(M_t(z_{\infty}))_{t \geq 0}$. Since $(M_t(z_{\infty}))_{t \geq 0}$ is only defined up to c\`adl\`ag
modifications (which are indistinguishable from each other), one can replace $(M_t(z_{\infty}))_{t \geq 0}$ 
by $ \left(M_t(z_{\infty}) \wedge Z_t \right)_{t \geq 0}$, and then suppose that for 
all $t \geq 0$, $M_t(z_{\infty}) \leq Z_t$ everywhere. 
Note that if $(Z_t)_{t \geq 0}$ is supposed to be uniformly integrable, $(M_t(z_{\infty}))_{t \geq 0}$
is also uniformly integrable, and since it tends $\mathbb{P}$-almost surely to zero, it is 
$\mathbb{P}$-almost surely identically zero.
This implies that $z_{\infty} = 0$, $\mathcal{Q}$-almost everywhere. 
Now, going back to the general case, let us define, for all $t \geq 0$:
$$\widetilde{Z}_t := Z_t - M_t(z_{\infty}).$$
Since $(M_t(z_{\infty}))_{t \geq 0}$ is a c\`adl\`ag $\mathbb{P}$-martingale, the process
$(\widetilde{Z}_t)_{t \geq 0}$ is a c\`adl\`ag, nonnegative $\mathbb{P}$-supermartingale.
Moreover, $M_t(z_{\infty})$ tends $\mathbb{P}$-almost surely to zero when $t$ goes to infinity, hence:
$$\widetilde{Z}_t \underset{t \rightarrow \infty}{\longrightarrow} Z_{\infty},$$
$\mathbb{P}$-almost surely. Now, since $(\widetilde{Z}_t)_{t \geq 0}$ is a nonnegative supermartingale 
and $Z_{\infty} \geq 0$, $\mathbb{P}$-almost surely, we 
obtain, for all $t \geq 0$:
\begin{equation}
0 \leq \mathbb{P} [Z_{\infty} | \mathcal{F}_t] \leq \widetilde{Z}_t, \label{nnnn}
\end{equation}
$\mathbb{P}$-almost surely. Hence, $\left( \left(\mathbb{P} [Z_{\infty} | \mathcal{F}_t] \right)_+ \wedge \widetilde{Z}_t 
\right)_{t \geq 0}$ is a c\`adl\`ag version of $(\mathbb{P} [Z_{\infty} | \mathcal{F}_t])_{t \geq 0}$ 
and one can suppose that (\ref{nnnn}) holds everywhere.
Now, let us write, for all $t \geq 0$:
$$\xi_t := \widetilde{Z}_t - \mathbb{P} [Z_{\infty} |\mathcal{F}_t].$$ 
This is nonnegative, c\`adl\`ag supermartingale
 tending $\mathbb{P}$-a.s. to zero when $t$ goes to infinity.
On the other hand, $\mathcal{Q}$-almost everywhere:
$$\underset{t \rightarrow \infty}{\lim} \, \frac{\xi_t}{X_t} = \underset{t \rightarrow \infty}{\lim} 
\, \frac{\widetilde{Z}_t}{X_t} = z_{\infty} - z_{\infty} = 0.$$
Here, the first equality is due to the fact that  
$(\mathbb{P} [Z_{\infty} |\mathcal{F}_t]/X_t)_{t \geq 0}$ tends to zero $\mathcal{Q}$-almost everywhere, by
 the remark made above on the case where $(Z_t)_{t \geq 0}$ is uniformly integrable. 
The uniqueness of the decomposition is very easy to check: since $M_t(z'_{\infty})$ and $\xi'_t$ tend
$\mathbb{P}$-almost surely to zero when $t \rightarrow \infty$, $Z_t$ tends $\mathbb{P}$-almost surely 
to $Z'_{\infty}$ and then $Z'_{\infty} = Z_{\infty}$. Similarly, since $\xi'_t /X_t$ and any 
c\`adl\`ag version of $\mathbb{P} [ Z'_{\infty} |\mathcal{F}_t] /X_t$ tend to 
zero, $\mathcal{Q}$-almost everywhere, 
$Z_t/X_t$ tends to $z'_{\infty}$, which is $\mathcal{Q}$-almost everywhere equal to 
$z_{\infty}$. One now deduces that for all $t \geq 0$, $\xi'_t = \xi_t$, $\mathbb{P}$-almost surely, and since
$\xi$ and $\xi'$ are c\`adl\`ag, they are indistinguishable, which proves the uniqueness of the decomposition. 
\end{proof}
As in \cite{NRY}, we can deduce, from Proposition \ref{decomposition}, the following characterization 
of the martingales of the form $(M_t(F))_{t \geq 0}$:
\begin{corollary}
Let us suppose that the assumptions of Theorem \ref{all} are satisfied, and that
 $A_{\infty} = \infty$, $\mathbb{P}$-almost surely. Then, a c\`adl\`ag, nonnegative $\mathbb{P}$-martingale
$(Z_t)_{t\geq 0}$ has the form $(M_t(F))_{t \geq 0}$ for a nonnegative, $\mathcal{Q}$-integrable
 functional
$F$, if and only if:
\begin{equation}
\mathbb{P}[Z_0] = \mathcal{Q} \left( \underset{t \rightarrow \infty}{\lim} \, \frac{Z_t}{X_t} \right). \label{pz0}
\end{equation}
Note that, by Proposition \ref{decomposition}, the limit above necessarily exists $\mathcal{Q}$-almost 
everywhere.
\end{corollary}
\begin{proof}
By Proposition \ref{decomposition}, one can write the decomposition:
$$Z_t = M_t(z_{\infty}) + \mathbb{P} [Z_{\infty} | \mathcal{F}_t] + \xi_t.$$
Note that in this situation, $(\xi_t)_{t \geq 0}$ is a nonnegative martingale. One has:
$$\mathbb{P}[Z_0] = \mathbb{P} [M_0 (z_{\infty})] + \mathbb{P} [ \, \mathbb{P}[ Z_{\infty} | \mathcal{F}_0] \, ]
+ \mathbb{P} [\xi_0] = \mathcal{Q} [z_{\infty}] + \mathbb{P} [Z_{\infty}] + \mathbb{P} [\xi_0].$$
Now, the equation \eqref{pz0} is satisfied iff
$$\mathbb{P} [Z_0] = \mathcal{Q} [z_{\infty}].$$
If this condition holds, one has 
$$\mathbb{P} [Z_{\infty}] = \mathbb{P} [\xi_0] = 0,$$
and then, for all $t \geq 0$, 
 $$\mathbb{P} [Z_{\infty} | \mathcal{F}_t] + \xi_t = 0$$ 
almost surely. Hence, the martingale $(Z_t)_{t \geq 0}$ is a c\`adl\`ag modification 
of $(M_t(z_{\infty}))_{t \geq 0}$. Since $(M_t(z_{\infty}))_{t \geq 0}$ is only defined up to 
c\`adl\`ag modification, one can suppose that $(Z_t)_{t \geq 0}$ coincides with $(M_t(z_{\infty}))_{t \geq 0}$. 
On the other hand, if $(Z_t)_{t \geq 0}$ has the form $(M_t(F))_{t \geq 0}$, by uniqueness of the decomposition
given in Proposition \ref{decomposition}, $F = z_{\infty}$, $\mathcal{Q}$-almost everywhere, which implies 
that $\mathbb{P}[Z_0]= \mathcal{Q}[z_{\infty}]$, and then \eqref{pz0} is satisfied. 
\end{proof}
 \begin{remark}
Let us suppose that, in Proposition \ref{decomposition}, $(Z_t)_{t \geq 0}$ is a nonnegative martingale.
Since the space satisfies the property (NP), there exists a unique finite measure $\mathbb{Q}_Z$ 
on $(\Omega, \mathcal{F})$, such that for all $t \geq 0$, its restriction to $\mathcal{F}_t$ has 
density $Z_t$ with respect to $\mathbb{P}$. If one writes the decomposition
$$Z_t = M_t(z_{\infty}) + \mathbb{P} [Z_{\infty} | \mathcal{F}_t] + \xi_t,$$
one deduces:
$$\mathbb{Q}_Z = z_{\infty} \,. \mathcal{Q} + Z_{\infty} \, . \mathbb{P} + \mathbb{Q}_{\xi},$$
where the restriction of $\mathbb{Q}_{\xi}$ to $\mathcal{F}_t$ has density $\xi_t$ with respect to 
$\mathbb{P}$. By Radon-Nykodym theorem, one has a decomposition:
$$\mathbb{Q}_{\xi} = \xi'. \mathbb{P} + \mathbb{Q}'_{\xi},$$
where $\mathbb{Q}'_{\xi}$ is singular with respect to $\mathbb{P}$. Now, if for $t \geq 0$, $\xi'_t$
is the density, with respect to $\mathbb{P}$, of the restriction of $\xi'.\mathbb{P}$ to $\mathcal{F}_t$,
then for all $t \geq 0$, $\xi'_t \leq \xi_t$, $\mathbb{P}$-almost surely. Moreover, if $(\xi'_t)_{t \geq 0}$ is
supposed to be c\`adl\`ag, then almost surely, $\xi'_t \leq \xi_t$ for all $t \geq 0$. By taking the 
$\mathbb{P}$-almost sure limit for $t$ going to infinity, one deduces that $\xi' = 0$, $\mathbb{P}$-almost
surely, therefore, 
$\mathbb{Q}_{\xi} = \mathbb{Q}'_{\xi}$ is singular with respect to $\mathbb{P}$. 
One can also decompose $\mathbb{Q}_{\xi}$ as:
$$\mathbb{Q}_{\xi} = \xi''. \mathcal{Q} + \mathbb{Q}''_{\xi},$$
where $\mathbb{Q}''_{\xi}$ is singular with respect to $\mathcal{Q}$. Now, for all $t \geq 0$, one has, 
$\mathbb{P}$-almost surely, and then $\mathcal{Q}$-almost everywhere, 
$M_t(\xi'') \leq \xi_t$. Since $(M_t(\xi''))_{t \geq 0}$ and $(\xi_t)_{t \geq 0}$ are right-continuous, 
one deduces that $\mathcal{Q}$-almost everywhere, $M_t(\xi'') \leq \xi_t$ for all $t \geq 0$. 
Since $\mathcal{Q}$-almost everywhere, $M_t(\xi'')/X_t$ tends to $\xi''$ when $t$ goes to infinity, and 
$\xi_t/X_t$ tends to zero, one has $\xi'' = 0$, $\mathcal{Q}$-almost everywhere and $\mathbb{Q}_{\xi} = \mathbb{Q}''_{\xi}$ is singular with
respect to $t$. Hence, we have obtained a decomposition of $\mathbb{Q}_Z$ into three parts:
\begin{itemize}
\item A part which is absolutely continuous with respect to $\mathbb{P}$.
\item A part which is absolutely continuous with respect to $\mathcal{Q}$. 
\item A part which is singular with respect to $\mathbb{P}$ and $\mathcal{Q}$.
\end{itemize}
\noindent
This decomposition is unique, as a consequence of uniqueness of Radon-Nykodym decomposition (recall that
 $\mathbb{P}$ and $\mathcal{Q}$ are mutually singular, since $A_{\infty} = \infty$, $\mathbb{P}$-almost
 surely, and $A_{\infty} < \infty$, $\mathcal{Q}$-almost everywhere). 
 This uniqueness can be compared with the uniqueness of the decomposition of the martingale 
$(Z_t)_{t \geq 0}$ given in Proposition \ref{decomposition}.
\end{remark}

\end{document}